\documentclass[a4paper]{amsart}
\usepackage{amsmath}
\usepackage{amssymb}
\usepackage{amsthm,amsxtra}
\usepackage{cite}
\usepackage{enumerate}
\usepackage{enumitem}
\usepackage[mathscr]{eucal}
\usepackage{adjustbox}
\usepackage{float}
\usepackage{tikz-cd}
\usepackage{tkz-graph}
\usepackage{tkz-berge}
\usetikzlibrary{positioning,arrows,shapes.geometric,trees}
\usepackage{graphicx}
\usepackage{mathtools}
\usetikzlibrary{positioning}
\usetikzlibrary{decorations.text}
\usepackage{refcount}
\newtheorem{Def}{Definition}[section]
\newtheorem{Th}[Def]{Theorem}
\newtheorem{Ex}[Def]{Example}

\newtheorem{Prop}[Def]{Proposition}
\newtheorem{Cor}[Def]{Corollary}

{}
\newtheorem{Prob}[Def]{Problem}

\DeclareMathOperator{\Sf}{S_{fin}}
\DeclareMathOperator{\Gf}{G_{fin}}

\DeclareMathOperator{\Uf}{U_{fin}}
\DeclareMathOperator{\Guf}{G_{ufin}}

\DeclareMathOperator{\SSS1}{SS_1^\ast}
\DeclareMathOperator{\SSf}{S_{fin}^\ast}
\DeclareMathOperator{\SGf}{G_{fin}^\ast}
\DeclareMathOperator{\SUf}{U_{fin}^\ast}
\DeclareMathOperator{\SGuf}{G_{ufin}^\ast}
\DeclareMathOperator{\SSSf}{SS_{fin}^\ast}
\DeclareMathOperator{\SSGf}{SG_{fin}^\ast}
\DeclareMathOperator{\SSSC}{SS_{comp}^\ast}
\DeclareMathOperator{\SSGC}{SG_{comp}^\ast}
\DeclareMathOperator{\SK}{SS_{\mathcal{K}}^\ast}

\DeclareMathOperator{\Int}{Int}
\DeclareMathOperator{\Cl}{Cl}

\begin{document}
\title[On certain star-$K$ variant of a selection principle]{On certain star-$K$ variant of a selection principle}

\author[ D. Chandra, N. Alam ]{ Debraj Chandra$^*$, Nur Alam$^*$
}
\newcommand{\acr}{\newline\indent}
\address{\llap{*\,}Department of Mathematics, University of Gour
Banga, Malda-732103, West Bengal, India}
\email{debrajchandra1986@gmail.com, nurrejwana@gmail.com}

\thanks{ The second author
is thankful to University Grants Commission (UGC), New
Delhi-110002, India for granting UGC-NET Junior Research
Fellowship (1173/(CSIR-UGC NET JUNE 2017)) during the tenure of
which this work was done.}


\maketitle

\begin{abstract}
This article deals with certain star variants of the Scheepers
property. We introduce and study the star-$K$-Scheepers property and corresponding game. The relationships between the game corresponding to the star-$K$-Scheepers property and other variants of the Scheepers games are investigated. Furthermore, the relation among the winning strategies of the players in such games are represented in an implication diagram. We present certain observations in terms of Alster covers and the Alexandroff duplicates. In addition, we examine few preservation like properties in this context. Some open problems are also posed.
\end{abstract}

\noindent{\bf\keywordsname{}:} {Scheepers property,
star-Scheepers, star-$K$-Scheepers, strongly star-Scheepers.}

\section{Introduction}
The study of selection principles and corresponding games in set-theoretic topology has a renowned history. After the systematic study of selection principles and corresponding games by Scheepers in \cite{coc1} (see also \cite{coc2}), it has acquired considerable importance in becoming one of the most active research fields. It has been observed that various topological notions can be defined or characterized in terms of selection principles and associated games. Since then the idea of selection principles has been generalized in many ways. In \cite{LjSM}, Ko\v{c}inac introduced a generalization of selection principles by defining star selection principles and he mainly investigated star versions of the Menger property \cite{coc1}. Readers interested in recent developments of star selection principles may consult the papers \cite{LjSM,BCS,BGM,BM,dcna21,dcna22-1,dcna22-2,SVM,survey,survey1} and references therein.

Note that one of the most important selection principle is
$\Uf(\mathcal{O},\Omega)$, nowadays called the Scheepers
property (\cite{coc1,coc2}, see also \cite{FPM} where different terminology is used). In general, Scheepers property is stronger than the Menger property and weaker than the Hurewicz property. The star variations and weaker forms of the Scheepers property have been recently investigated in \cite{dcna22-2,dcna22-3}. In this article we introduce another star variant of the Scheepers property, namely star-$K$-Scheepers property.

The paper is organised as follows. In Section 3, we mainly introduce the star-$K$-Scheepers property. The interrelationships between the properties considered here are outlined into an implication diagram (Figure~\ref{dig1}). Certain situations are discussed when all the star versions of the Scheepers property considered here are identical. We present few observations using Alster covers. It is observed that the extent of a star-$K$-Scheepers space can be made as large as possible. Later in this section, we present game theoretic observations corresponding to the selection principles considered here. In the context of paracompact spaces we show that the game corresponding to the star-$K$-Scheepers property is equivalent to some other related games. We also show that the role of paracompactness is essential in this setting. We sketch another implication diagram (Figure~\ref{dig2}) to describe the relationship between the winning strategies in the games considered here. In Section 4, certain observations on the Alexandroff duplicates (\cite{AD,Engelking}) are presented. Further we investigate preservation kind of properties under certain topological operations. In the final section, some open problems are posed.

\section{Preliminaries}
Throughout the paper $(X,\tau)$ stands for a topological space.
For undefined notions and terminologies see \cite{Engelking}. Let $\mathcal{A}$ and $\mathcal{B}$ be collections of families of subsets of an infinite space $X$. Following \cite{coc1,coc2}, we define

\noindent $\Sf(\mathcal{A},\mathcal{B})$: For each sequence
$(\mathcal{U}_n)_{n\in\mathbb{N}}$ of elements of $\mathcal{A}$ there exists a sequence $(\mathcal{V}_n)_{n\in\mathbb{N}}$ such that for each $n\in\mathbb{N}$, $\mathcal{V}_n$ is a finite subset of $\mathcal{U}_n$ and
$\cup_{n\in\mathbb{N}}\mathcal{V}_n\in\mathcal{B}$.\\

\noindent $\Uf(\mathcal{A},\mathcal{B})$: For each sequence
$(\mathcal{U}_n)_{n\in\mathbb{N}}$ of elements of $\mathcal{A}$ there exists a sequence $(\mathcal{V}_n)_{n\in\mathbb{N}}$ such that for each $n\in\mathbb{N}$, $\mathcal{V}_n$ is a finite subset of $\mathcal{U}_n$ and
$\{\cup\mathcal{V}_n : n\in\mathbb{N}\}\in\mathcal{B}$ or there
is some $n\in\mathbb{N}$ such that $\cup\mathcal{V}_n=X$.

For a subset $A$ of a space $X$ and a collection $\mathcal{P}$
of subsets of $X$, $St(A,\mathcal{P})$ denotes the star of $A$
with respect to $\mathcal{P}$, that is the set
$\cup\{B\in\mathcal{P} : A\cap B\neq\emptyset\}$. For $A=\{x\}$,
$x\in X$, we write $St(x,\mathcal{P})$ instead of
$St(\{x\},\mathcal{P})$ \cite{Engelking}.

In \cite{LjSM}, Ko\v{c}inac introduced star selection principles in the following way.

\noindent $\SSf(\mathcal{A},\mathcal{B})$: For each sequence
$(\mathcal{U}_n)_{n\in\mathbb{N}}$ of elements of $\mathcal{A}$ there exists a sequence $(\mathcal{V}_n)_{n\in\mathbb{N}}$ such that for each $n\in\mathbb{N}$, $\mathcal{V}_n$ is a finite subset of $\mathcal{U}_n$ and
$\cup_{n\in\mathbb{N}}\{St(V,\mathcal{U}_n) : V\in\mathcal{V}_n\}\in\mathcal{B}$.\\

\noindent $\SUf(\mathcal{A},\mathcal{B})$: For each sequence
$(\mathcal{U}_n)_{n\in\mathbb{N}}$ of elements of $\mathcal{A}$ there exists a sequence $(\mathcal{V}_n)_{n\in\mathbb{N}}$ such that for each $n\in\mathbb{N}$, $\mathcal{V}_n$ is a finite subset of $\mathcal{U}_n$ and
$\{St(\cup\mathcal{V}_n,\mathcal{U}_n) :
n\in\mathbb{N}\}\in\mathcal{B}$ or there is some $n\in\mathbb{N}$ such that
$St(\cup\mathcal{V}_n,\mathcal{U}_n)=X$.

Let $\mathcal{K}$ be a collection of subsets of $X$. Then we define

\noindent $\SK(\mathcal{A},\mathcal{B})$: For each sequence
$(\mathcal{U}_n)_{n\in\mathbb{N}}$ of elements of $\mathcal{A}$ there exists a sequence $(K_n)_{n\in\mathbb{N}}$ of elements of $\mathcal{K}$ such that
$\{St(K_n,\mathcal{U}_n) : n\in\mathbb{N}\}\in\mathcal{B}$.
If $\mathcal{K}$ is the collection of one-point (resp. finite, compact)
subsets of $X$, then we use $\SSS1(\mathcal{A},\mathcal{B})$ (resp. $\SSSf(\mathcal{A},\mathcal{B})$, $\SSSC(\mathcal{A},\mathcal{B})$) instead of $\SK(\mathcal{A},\mathcal{B})$.

Let $\mathcal O$ denote the collection of all open covers of
$X$. An open cover $\mathcal{U}$ of $X$ is said to be a $\gamma$-cover if each element of $X$ does not belong to at most finitely many members of $\mathcal{U}$ \cite{survey} (see also \cite{coc2,coc1}). We use the symbol $\Gamma$ to denote the collection of all $\gamma$-covers of $X$. An open cover $\mathcal{U}$ of $X$ is said to be an $\omega$-cover if for each finite subset $F$ of $X$ there is a set $U\in\mathcal{U}$ such that $F\subseteq U$ \cite{coc1,coc2}. We use the symbol $\Omega$ to denote the collection of all
$\omega$-covers of $X$. Note that $\Gamma\subseteq\Omega\subseteq\mathcal{O}$. An open cover $\mathcal{U}$ of $X$ is said to be weakly groupable \cite{cocVIII} if it can be expressed as a countable union of finite, pairwise disjoint subfamilies $\mathcal{U}_n$,
$n\in\mathbb{N}$, such that for each finite set $F\subseteq X$ we have $F\subseteq\cup\mathcal{U}_n$ for some $n\in\mathbb{N}$. The collection of all weakly groupable open covers of $X$ is denoted by $\mathcal{O}^{wgp}$.

A space $X$ is said to have the Menger (resp. Scheepers, Hurewicz) property if it satisfies the selection hypothesis $\Sf(\mathcal{O},\mathcal{O})$ (resp. $\Uf(\mathcal{O},\Omega)$, $\Uf(\mathcal{O},\Gamma)$) \cite{coc1,coc2} (see also \cite{FPM}).

A space $X$ is said to have the (1) star-Menger property, (2)
star-$K$-Menger property, (3) strongly star-Menger property, (4)
star-Hurewicz property, (5) star-$K$-Hurewicz property and (6)
strongly star-Hurewicz property if it satisfies the selection
hypothesis (1) $\SSf(\mathcal{O},\mathcal{O})$, (2)
$\SSSC(\mathcal{O},\mathcal{O})$, (3)
$\SSSf(\mathcal{O},\mathcal{O})$, (4)
$\SUf(\mathcal{O},\Gamma)$, (5) $\SSSC(\mathcal{O},\Gamma)$ and
(6) $\SSSf(\mathcal{O},\Gamma)$ respectively \cite{LjSM,sHR}
(see also \cite{LjSM-II,rsM,rssM,sKH,sKM}).

A space $X$ is said to be starcompact (resp. star-Lindel\"{o}f)
if for every open cover $\mathcal{U}$ of $X$ there exists a
finite (resp. countable) $\mathcal{V}\subseteq\mathcal{U}$ such
that $St(\cup\mathcal{V},\mathcal{U})=X$. $X$ is said to be
strongly starcompact (resp. $K$-starcompact, strongly
star-Lindel\"{o}f) if for every open cover $\mathcal{U}$ of $X$
there exists a finite (resp. compact, countable) set $K\subseteq
X$ such that $St(K,\mathcal{U})=X$ \cite{LjSM,sCP,sKM}.
A subset $A$ of a space $X$ is said to be regular-closed in $X$
if $\Cl(\Int A)=A$.
For a Tychonoff space $X$, $\beta X$ denotes the \v{C}ech-Stone
compactification of $X$.
For a space $X$, $e(X)=\sup\{|Y| : Y\;\text{is a closed and
discrete subspace of}\; X\}$ is said to be the extent of $X$.

A natural pre-order $\leq^*$ on the Baire space
$\mathbb{N}^\mathbb{N}$ is defined by $f\leq^*g$ if and only if
$f(n)\leq g(n)$ for all but finitely many $n$. A subset $D$ of $\mathbb{N}^\mathbb{N}$ is said to be dominating
if for each $g\in\mathbb{N}^\mathbb{N}$ there exists a $f\in D$
such that $g\leq^* f$. A subset $A$ of $\mathbb{N}^\mathbb{N}$ is said to be bounded if there is a $g\in\mathbb{N}^\mathbb{N}$ such that $f\leq^*g$ for all $f\in A$. The minimum cardinality of a dominating subset of $\mathbb{N}^\mathbb{N}$ is denoted by $\mathfrak{d}$, the minimum cardinality of an unbounded subset of $\mathbb{N}^\mathbb{N}$ is denoted by $\mathfrak{b}$ and $\mathfrak{c}$ denotes the cardinality of the set of the real numbers. A infinite set $r\subseteq\mathbb{R}$ reaps a family $\mathcal{A}$ of countable subsets of $\mathbb{N}$ if for each set $a\in\mathcal{A}$ both sets $a\cap r$ and $a\setminus r$ are infinite. The minimum cardinality of a family $\mathcal{A}$ of countable subsets of $\mathbb{N}$ that no set reaps is denoted by $\mathfrak{r}$. For any cardinal $\kappa$, $\kappa^+$ denotes the smallest cardinal greater than $\kappa$.

\section{The star-$K$-Scheepers property and related games}
\subsection{The star-$K$-Scheepers property}
In \cite{dcna22-2}, we introduced the following two star
variants of the Scheepers property.
\begin{Def}
\label{D1}
A space $X$ is said to have the star-Scheepers property if it
satisfies the selection hypothesis $\SUf(\mathcal{O},\Omega)$.
\end{Def}

\begin{Def}
\label{D2}
A space $X$ is said to have the strongly star-Scheepers property
if it satisfies the selection hypothesis
$\SSSf(\mathcal{O},\Omega)$.
\end{Def}

We now introduce the main definition of this paper.
\begin{Def}
\label{D3}
A space $X$ is said to have the star-$K$-Scheepers property if
it satisfies the selection hypothesis
$\SSSC(\mathcal{O},\Omega)$.
\end{Def}
We call a space $X$ star-$K$-Scheepers if $X$ has the
star-$K$-Scheepers property.
\begin{figure}[h]
\begin{adjustbox}{max width=\textwidth,max
height=\textheight,keepaspectratio,center}
\begin{tikzcd}[column sep=4ex,row sep=6ex,arrows={crossing
over}]
\SUf(\mathcal{O},\Gamma)\arrow[rr] &&
\SUf(\mathcal{O},\Omega)\arrow[rr] &&
\SSf(\mathcal{O},\mathcal{O})
\\
\SSSC(\mathcal{O},\Gamma)\arrow[rr]\arrow[u]&&
\SSSC(\mathcal{O},\Omega) \arrow[rr]\arrow[u]&&
\SSSC(\mathcal{O},\mathcal{O})\arrow[u]
\\
\SSSf(\mathcal{O},\Gamma) \arrow[rr]\arrow[u]&&
\SSSf(\mathcal{O},\Omega) \arrow[rr]\arrow[u]&&
\SSSf(\mathcal{O},\mathcal{O}) \arrow[u]
\\
 \Uf(\mathcal{O},\Gamma) \arrow[rr]\arrow[u]&&
\Uf(\mathcal{O},\Omega)\arrow[rr]\arrow[u]&&
\Sf(\mathcal{O},\mathcal{O})\arrow[u]
\end{tikzcd}
\end{adjustbox}
\caption{Star variations of the Hurewicz, Scheepers and Menger
properties}
\label{dig1}
\end{figure}

Next we discuss some examples to make distinction between the
properties considered here.

Let $aD$ be the one point compactification of the discrete space
$D$ with $|D|=\mathfrak c$. Also consider the subspace
$X=(aD\times [0,\mathfrak{c}^+))\cup
(D\times\{\mathfrak{c}^+\})$ of the product space $aD\times
[0,\mathfrak{c}^+]$. By \cite[Example 2.2]{sKM}, $X$ is a
Tychonoff $K$-starcompact space which is not strongly
star-Menger. This shows the existence of a star-$K$-Scheepers
space which is not strongly star-Scheepers.

\begin{Ex}
\label{E3}
\emph{There is a Hausdorff star-Scheepers space which is not
star-$K$-Scheepers.}\\
Let $P=\{x_\alpha : \alpha<\mathfrak{c}\}$ and $Q=\{y_n :
n\in\mathbb{N}\}$. Define $Y=\{(x_\alpha, y_n) :
\alpha<\mathfrak{c}, n\in\mathbb{N}\}$ and $X=Y\cup P\cup\{p\}$,
where $p\notin Y\cup P$. We now define a topology on $X$ as
follows. Every point of $Y$ is isolated, a basic neighbourhood
of a point $x_\alpha\in P$ for each $\alpha<\mathfrak{c}$ is of
the form $$U_{x_\alpha}(n)=\{x_\alpha\}\cup\{(x_\alpha, y_m) :
m>n\}$$ for $n\in\mathbb{N}$ and a basic neighbourhood of $p$ is
of the form $$U_p(A)=\{p\}\cup\{(x_\alpha, y_n) : x_\alpha\in
P\setminus A, n\in\mathbb{N}\}$$ for a countable subset $A$ of
$P$. In \cite[Example 2.7]{sKH}, it is proved that $X$ is a Hausdorff star-Hurewicz space, and hence $X$ is star-Scheepers. In \cite[Example 2.3]{RSKM}, it is proved that $X$ is not star-$K$-Menger, hence not star-$K$-Scheepers.
\end{Ex}

Next result follows from \cite[Theorem 2.11]{LjSM} with necessary
modifications.
\begin{Prop}
Every paracompact star-$K$-Scheepers space is Scheepers.
\end{Prop}

Similar to the other classical selective properties, the
Scheepers property is also equivalent to any of its star
variants in paracompact Hausdorff spaces.
\begin{Prop}
\label{T1}
For a paracompact Hausdorff space $X$ the following assertions
are equivalent.
\begin{enumerate}[label={\upshape(\arabic*)},
leftmargin=*]
\item $X$ is Scheepers.

\item $X$ is strongly star-Scheepers.

\item $X$ is star-$K$-Scheepers.

\item $X$ is star-Scheepers.
\end{enumerate}
\end{Prop}

By \cite[Theorem 8.10]{SFPH}, there is a set  of reals $X$ which
is Scheepers but not Hurewicz. With the help of Proposition~\ref{T1} and \cite[Proposition 4.1]{sHR}, we can conclude that there exists a star-$K$-Scheepers space which is not star-$K$-Hurewicz. Also by \cite[Theorem 2.1]{FPM}, assuming $\mathfrak{d}\leq\mathfrak{r}$, there is a set of reals $X$ which is Menger but not Scheepers. Thus Proposition~\ref{T1} and \cite[Theorem 2.8]{LjSM} together imply the existence of a star-$K$-Menger space which is not star-$K$-Scheepers.

\begin{Prob}
\label{P6}
Is there a ZFC example of a star-$K$-Menger space which is not star-$K$-Scheepers?
\end{Prob}

\begin{Th}
\label{T17}
If each finite power of $X$ is star-$K$-Menger, then $X$ is
star-$K$-Scheepers.
\end{Th}
\begin{proof}
We pick a sequence $(\mathcal{U}_n)_{n\in\mathbb{N}}$ of open covers of $X$ to show that $X$ is star-$K$-Scheepers. Let $\{N_k : k\in\mathbb{N}\}$ be a partition of $\mathbb{N}$ into infinite subsets. For each $k\in\mathbb{N}$ and each $m\in N_k$ choose $$\mathcal{W}_m=\{U_1\times U_2\times\dotsb\times U_k : U_1,U_2,\dotsc,U_k\in\mathcal{U}_m\}.$$ Observe that $(\mathcal{W}_m:{m\in N_k})$ is a sequence of open covers of $X^k$. Since $X^k$ is star-$K$-Menger, there exists a sequence $(K_m:{m\in N_k})$ of compact subsets of $X^k$ such that $\{St(K_m,\mathcal{W}_m) : m\in N_k\}$ is an open cover of $X^k$. For each $1\leq i\leq k$ let $p_i:X^k\to X$ be the $i$th projection mapping. Clearly $p_i(K_m)$ is a compact subset of $X$ for each $1\leq i\leq k$ and each $m\in N_k$. Next for each $m\in N_k$ we consider the subset $$C_m=\cup_{1\leq i\leq k}p_i(K_m)$$ of $X$. Clearly $C_m$ is compact and $K_m\subseteq C_m^k$ for each $m\in N_k$. Thus the sequence $(C_n)_{n\in\mathbb{N}}$ witnesses for $(\mathcal{U}_n)_{n\in\mathbb{N}}$ that $X$ is star-$K$-Scheepers.
\end{proof}

For any two collections $\mathcal{A}$ and $\mathcal{B}$ of subsets of a space $X$, we denote by $\mathcal{A}\wedge\mathcal{B}$ the set $\{A\cap B : A\in\mathcal{A}, B\in\mathcal{B}\}$.
\begin{Th}
\label{T18}
For a space $X$ the following assertions are equivalent.
\begin{enumerate}[label={\upshape(\arabic*)}, leftmargin=*]
  \item $X$ is star-$K$-Scheepers.
  \item $X$ satisfies $\SSSC(\mathcal{O},\mathcal{O}^{wgp})$.
\end{enumerate}
\end{Th}
\begin{proof}
Since each countable $\omega$-cover is weakly groupable, $(1)$ implies $(2)$ is trivial, so that we have to prove only $(2)\Rightarrow(1)$. We choose a sequence $(\mathcal{U}_n)_{n\in\mathbb{N}}$ of open covers of $X$ to prove that $X$ is star-$K$-Scheepers. For each $n\in\mathbb{N}$, $\mathcal{W}_n=\wedge_{i\leq n}\mathcal{U}_i$ is an open cover of $X$ which refines $\mathcal{U}_i$ for each $i\leq n$. Apply $(2)$ to the sequence $(\mathcal{W}_n)_{n\in\mathbb{N}}$ to obtain a sequence $(K_n)_{n\in\mathbb{N}}$ of compact subsets of $X$ such that $$\{St(K_n,\mathcal{W}_n) : n\in\mathbb{N}\}\in\mathcal{O}^{wgp}.$$  This gives us a sequence $n_1<n_2<\cdots<n_k<\cdots$ of positive integers such that for every finite set $F\subseteq X$ we have $$F\subseteq\cup_{n_k\leq i< n_{k+1}}St(K_i,\mathcal{W}_i)$$ for some $k\in\mathbb{N}$. Next we define a sequence $(K_n^\prime)_{n\in\mathbb{N}}$ of compact subsets of $X$ as follows.
\[
K_n^\prime=
\begin{cases}
 \cup_{i<n_1}K_i & \mbox{for } n<n_1  \\
  \cup_{n_k\leq i<n_{k+1}}K_i & \mbox{for } n_k\leq n<n_{k+1}.
\end{cases}
\]

Observe that $(K_n^\prime)_{n\in\mathbb{N}}$ is the required sequence for $(\mathcal{U}_n)_{n\in\mathbb{N}}$ that $X$ is star-$K$-Scheepers.
\end{proof}

We recall the following classes of covers from \cite{Alster}.
\begin{enumerate}[leftmargin=*]
\item[$\mathcal{G}$:] The family of all covers $\mathcal{U}$ of the space $X$ for which each element of $\mathcal{U}$ is a $G_\delta$ set.

\item[$\mathcal{G}_K$:] The family of all Alster covers of $X$. A cover $\mathcal{U}\in\mathcal{G}$ is said to be an Alster cover if $X$ is not in $\mathcal{U}$ and for each compact set $C\subseteq X$ there is a $U\in\mathcal{U}$ such that $C\subseteq U$.

\item[$\mathcal{G}_\Gamma$:] The family of all covers $\mathcal{U}\in\mathcal{G}$ which are infinite and each infinite subset of $\mathcal{U}$ is a cover of $X$.
\end{enumerate}

From Theorem~\ref{T5} to Theorem~\ref{T6} we assume that all considered covers are closed under finite unions.
\begin{Th}
\label{T5}
If a space $X$ has the property $\SSS1(\mathcal{G}_K,\mathcal{G})$, then $X$ is strongly star-Scheepers.
\end{Th}
\begin{proof}
Consider a sequence $(\mathcal{U}_n)_{n\in\mathbb{N}}$ of open covers of $X$. Let $\{N_k : k\in\mathbb{N}\}$ be a partition of $\mathbb{N}$ into infinite sets. Fix $k\in\mathbb{N}$. For each $n\in N_k$ let $\mathcal{W}_n=\{U^k : U\in\mathcal{U}_n\}$. Of course $(\mathcal{W}_n : n\in N_k)$ is a sequence of open covers of $X^k$. Define $\mathcal{U}=\{\cap_{n\in N_k}W_n : W_n\in\mathcal{W}_n\}$. Observe that $\mathcal{U}\in\mathcal{G}_K$ for $X^k$. Without loss of generality we suppose that $\mathcal{U}=\mathcal{H}_1\times\mathcal{H}_2\times\cdots\times\mathcal{H}_k$, where for each $1\leq i\leq k$ $\mathcal{H}_i\in\mathcal{G}_K$ for $X$. Since $X$ satisfies $\SSS1(\mathcal{G}_K,\mathcal{G})$, for each $1\leq i\leq k$ we obtain a countable set $C_i\subseteq X$ such that $\{St(x,\mathcal{H}_i) : x\in C_i\}\in\mathcal{G}_K$ for $X$ and hence there is a countable set $C=C_1\times C_2\times\cdots\times C_k\subseteq X^k$ such that $St(C,\mathcal{U})$ covers $X^k$. Enumerate $C$ as $\{(x_1^{(n)},x_2^{(n)},\dotsc,x_k^{(n)}) : n\in N_k\}$. For each $n\in N_k$ we choose $F_n=\{x_i^{(n)} : 1\leq i\leq k\}$. The sequence $(F_n)_{n\in\mathbb{N}}$ now witnesses for $(\mathcal{U}_n)_{n\in\mathbb{N}}$ that $X$ is strongly star-Scheepers.
\end{proof}

\begin{Cor}
\label{C6}
If a space $X$ has the property $\SSS1(\mathcal{G}_K,\mathcal{G})$, then $X$ is star-$K$-Scheepers.
\end{Cor}

\begin{Th}
\label{T6}
If $X$ satisfies $\SSS1(\mathcal{G}_K,\mathcal{G}_\Gamma)$ and $Y$ is star-$K$-Scheepers, then $X\times Y$ is also star-$K$-Scheepers.
\end{Th}
\begin{proof}
Let $(\mathcal{U}_n)_{n\in\mathbb{N}}$ be a sequence of open covers of $X\times Y$. Then there are two sequences $(\mathcal{A}_n)_{n\in\mathbb{N}}$ and $(\mathcal{B}_n)_{n\in\mathbb{N}}$ of open covers of $X$ and $Y$ respectively such that if $A\in\mathcal{A}_n$ (resp. $B\in\mathcal{B}_n$), there exist a $B\in\mathcal{B}_n$ (resp. a $A\in\mathcal{A}_n$) and a $U\in\mathcal{U}_n$ such that $A\times B\subseteq U$, and if $U\in\mathcal{U}_n$, then there exist a $A\in\mathcal{A}_n$ and a $B\in\mathcal{B}_n$ such that $A\times B\subseteq U$.

Let $K$ be a compact subset of $X$. For each $y\in Y$ and for each $n\in\mathbb{N}$ there exists a $U\in\mathcal{U}_n$ such that $K\times\{y\}\subseteq U$. Let $V$ be an open set in $X$ containing $K$ such that $K\times\{y\}\subseteq V\times\{y\}\subseteq U$. For each $n\in\mathbb{N}$ we get a finite subset $\{A_i^{(n)} : 1\leq i\leq k_n\}$ of $\mathcal{A}_n$ such that $K\subseteq\cup_{1\leq i\leq k_n}A_i^{(n)}=V_n$ (say). It follows that $K\subseteq V_n\cap V$ for all $n\in\mathbb{N}$. Choose $\Phi(K)=\cap_{n\in\mathbb{N}}(V_n\cap V)$. Thus for each compact subset $K$ of $X$ we obtain a $G_\delta$ subset $\Phi(K)$ of $X$ such that $K\subseteq\Phi(K)$. Also we can say that for each compact subset $K$ of $X$, for each $y\in Y$ and for each $n\in\mathbb{N}$ there exists a $U\in\mathcal{U}_n$ such that $\Phi(K)\times\{y\}\subseteq U$.

For each $n\in\mathbb{N}$, $\mathcal{W}_n=\{\Phi(K) : K\text{ is a compact subset of }X\}$ is a member of $\mathcal{G}_K$. Since $X$ satisfies $SS_1^*(\mathcal{G}_K,\mathcal{G}_\Gamma)$, there exists a sequence $(x_n)_{n\in\mathbb{N}}$ of elements of $X$ such that $\{St(x_n,\mathcal{W}_n) : n\in\mathbb{N}\}\in\mathcal{G}_\Gamma$ for $X$. Observe that for each $n\in\mathbb{N}$, $\mathcal{O}_n=\{O\subseteq Y : O\text{ is open in }Y\text{ and }\{x_n\}\times O\subseteq U\text{ for some }U\in\mathcal{U}_n\}$ is an open cover of $Y$. Since $Y$ is star-$K$-Scheepers, we obtain a sequence $(C_n)_{n\in\mathbb{N}}$ of compact subsets of $Y$ such that $\{St(C_n,\mathcal{O}_n) : n\in\mathbb{N}\}$ is an $\omega$-cover of $Y$. For each $n\in\mathbb{N}$ let $F_n=\{x_n\}\times C_n$. We now show that the sequence $(F_n)_{n\in\mathbb{N}}$ witnesses for $(\mathcal{U}_n)_{n\in\mathbb{N}}$ that $X\times Y$ is star-$K$-Scheepers. Let $F$ be a finite subset of $X\times Y$ and choose finite sets $F_1\subseteq X$ and $F_2\subseteq Y$ such that $F\subseteq F_1\times F_2$. Subsequently $F_1\subseteq St(x_n,\mathcal{W}_n)$ for all but finitely many $n\in\mathbb{N}$ and $F_2\subseteq St(C_n,\mathcal{O}_n)$ for infinitely many $n\in\mathbb{N}$. Choose a $n_0\in\mathbb{N}$ so that $F_1\subseteq St(x_{n_0},\mathcal{W}_{n_0})$ and $F_2\subseteq St(C_{n_0},\mathcal{O}_{n_0})$. It now remains to show that $F\subseteq St(F_{n_0},\mathcal{U}_{n_0})$. Let $(x,y)\in F_1\times F_2$. Choose a compact set $K\subseteq X$ such that $x, x_{n_0}\in\Phi(K)$. Choose also a set $O\in\mathcal{O}_{n_0}$ containing $y$ and a set $U_1\in\mathcal{U}_{n_0}$ such that $O\cap C_{n_0}\neq\emptyset$ and $\{x_{n_0}\}\times O\subseteq U_1$. Observe that $U_1\cap F_{n_0}\neq\emptyset$. Also we can find a set $U_2\in\mathcal{U}_{n_0}$ such that $\Phi(K)\times\{y\}\subseteq U_2$. Thus there is a set $U=U_1\cup U_2\in\mathcal{U}_{n_0}$ such that $(x,y)\in U$ and $U\cap F_{n_0}\neq\emptyset$. Consequently $(x,y)\in St(F_{n_0},\mathcal{U}_{n_0})$ and $F\subseteq F_1\times F_2\subseteq St(F_{n_0},\mathcal{U}_{n_0})$ and the proof is now complete.
\end{proof}

We now exhibit examples to observe the behaviour of extent of star-$K$-Scheepers spaces.
\begin{Ex}
\label{E6}
\emph{For any infinite cardinal $\kappa$ there exists a Tychonoff star-$K$-Scheepers strongly star-Lindel\"{o}f space $X(\kappa)$ such that $e(X(\kappa))\geq\kappa$.}\\
Choose $D=\{f_\alpha : \alpha<\kappa\}$, where for each $\alpha<\kappa$, $f_\alpha\in\{0,1\}^\kappa$ such that
\begin{equation*}
f_\alpha(\beta)=
\begin{cases}
1 & \text{if~} \beta=\alpha\\
0 & \text{otherwise}.
\end{cases}
\end{equation*}
Consider $$X(\kappa)=(\{0,1\}^\kappa\times[0,\kappa^+])\setminus ((\{0,1\}^\kappa\setminus D)\times\{\kappa^+\})$$ as a subspace of the product space $\{0,1\}^\kappa\times[0,\kappa^+]$. The space $X(\kappa)$ is Tychonoff strongly star-Lindel\"{o}f and $D\times\{\kappa^+\}$ is a closed and discrete subset of it (see \cite[Theorem 1]{HWWE}). Thus $e(X(\kappa))\geq\kappa$.

We now show that $X(\kappa)$ is $K$-starcompact (hence star-$K$-Scheepers). Let $\mathcal{U}$ be an open cover of $X(\kappa)$. There is a $\gamma<\kappa^+$ such that for each $\alpha<\kappa$ we get an open set $U_\alpha$ in $\{0,1\}^\kappa$ containing $f_\alpha\in D$ and the set $$(U_\alpha\times(\gamma,\kappa^+])\cap X(\kappa)$$ is contained in some member of $\mathcal{U}$. Subsequently we can obtain a $\gamma<\beta<\kappa^+$ such that $\{0,1\}^\kappa\times[0,\beta]$ is compact and $$D\times\{\kappa^+\}\subseteq St(\{0,1\}^\kappa\times[0,\beta],\mathcal{U}).$$ Observe that $\{0,1\}^\kappa\times[0,\kappa^+)$ is $K$-starcompact since it is countably compact. This gives us a compact subspace $K$ of $X(\kappa)$ such that $$\{0,1\}^\kappa\times[0,\kappa^+)\subseteq St(K,\mathcal{U}).$$ The set $(\{0,1\}^\kappa\times[0,\beta])\cup K$ guarantees for $\mathcal{U}$ that $X(\kappa)$ is $K$-starcompact.
\end{Ex}

\begin{Ex}
\label{E5}
\emph{For any cardinal $\kappa>\omega$ there exists a Tychonoff star-$K$-Scheepers space $Y(\kappa)$ such that $e(Y(\kappa))\geq\kappa$ which is not strongly star-Lindel\"{o}f.}\\
Let $D=\{d_\alpha : \alpha<\kappa\}$ be the discrete space of cardinality $\kappa$. Consider $$Y(\kappa)=(\beta D\times[0,\kappa^+])\setminus((\beta D\setminus D)\times\{\kappa^+\})$$ as a subspace of $\beta D\times[0,\kappa^+]$. The space $Y(\kappa)$ is Tychonoff star-$K$-Scheepers (see \cite[Lemma 2.3]{sKH}). Observe that $D\times\{\kappa^+\}$ is a closed discrete subset of $Y(\kappa)$ and hence $e(Y(\kappa))\geq\kappa$.

We now show that $Y(\kappa)$ is not strongly star-Lindel\"{o}f.  Choose $$\mathcal{U}=\{\beta D\times[0,\kappa^+)\}\cup\{\{d_\alpha\}\times[0,\kappa^+] : \alpha<\kappa\}$$ and clearly $\mathcal{U}$ is an open cover of $Y(\kappa)$. Let $A$ be any countable subset of $Y(\kappa)$. Subsequently we get a $\gamma<\kappa$ with $$A\cap(\{d_\alpha\}\times[0,\kappa^+])=\emptyset$$ for all $\alpha>\gamma$. If we choose a $\gamma<\beta<\kappa$, then $(d_\beta,\kappa^+)\notin St(A,\mathcal{U})$ since $\{d_\beta\}\times[0,\kappa^+]$ is the only member of $\mathcal{U}$ containing the point $( d_\beta,\kappa^+)$. Thus $Y(\kappa)$ is not strongly star-Lindel\"{o}f.
\end{Ex}

We call a space $X$ metacompact (resp. subparacompact) \cite{Burke} if every open cover of it has a point-finite open refinement (resp. $\sigma$-discrete closed refinement). The spaces $X(\kappa)$ and $Y(\kappa)$ (in the preceding examples) are neither metacompact nor subparacompact since they have a non-compact countably compact closed subspace homeomorphic to $\kappa^+$. The following question may immediately arise.

\begin{Prob}
\label{Q1}
Can the extent of a metacompact (or, subparacompact) star-$K$-Scheepers space be arbitrarily large?
\end{Prob}

\subsection{Game theoretic observations}
Following \cite{coc1,LjSM}, we consider infinitely long games $\mathsf{\bf G}_{\prod}(\mathcal{O},\mathcal{B})$ corresponding to the selection principles ${\textstyle \prod}(\mathcal{O},\mathcal{B})$, where
\begin{eqnarray*}
{\textstyle \prod}&\in&\{\Sf,\Uf,\SSf,\SUf,\SSSf,\SSSC\},\\
\mathsf{\bf G}_{\prod}&\in&\{\Gf,\Guf,\SGf,\SGuf,\SSGf,\SSGC\}
\end{eqnarray*} and $\mathcal{B}\in\{\mathcal{O},\Omega,\Gamma\}$.

The game $\Guf(\mathcal{O},\Omega)$ on a space $X$ corresponding to the selection principle $\Uf(\mathcal{O},\Omega)$ is played as follows. Players ONE and TWO play an inning for each positive integer $n$. In the $n$th inning ONE chooses an open cover $\mathcal{U}_n$ of $X$ and TWO responds by selecting a finite subset $\mathcal{V}_n$ of $\mathcal{U}_n$. TWO wins the play $\mathcal{U}_1, \mathcal{V}_1, \mathcal{U}_2, \mathcal{V}_2,\ldots, \mathcal{U}_n, \mathcal{V}_n,\ldots$ of this game if $\{\cup\mathcal{V}_n : n\in\mathbb{N}\}$ is an $\omega$-cover of $X$ or there is some $n\in\mathbb{N}$ such that $X=\cup\mathcal{V}_n$; otherwise ONE wins. The game $\SSGC(\mathcal{O},\Omega)$ on a space $X$ corresponding to the selection principle $\SSSC(\mathcal{O},\Omega)$ is played as follows. Players ONE and TWO play an inning for each positive integer $n$. In the $n$th inning ONE chooses an open cover $\mathcal{U}_n$ of $X$ and TWO responds by selecting a compact subset $K_n$ of $X$. TWO wins the play $\mathcal{U}_1, K_1, \mathcal{U}_2, K_2,\dotsc, \mathcal{U}_n, K_n,\dotsc$ of this game if $\{St(K_n,\mathcal{U}_n) : n\in\mathbb{N}\}$ is an $\omega$-cover of $X$; otherwise ONE wins. Other games can be similarly defined.

It is easy to see that if ONE does not have a winning strategy in any of the above game on $X$, then $X$ satisfies the selection principle corresponding to that game. Recall that two games are said to be equivalent if whenever one
of the players has a winning strategy in one of the games, then that same player has a winning strategy in the other game \cite{Pearl}.

\begin{Th}[{\cite[Theorem 5.1]{dcna22-2}}]
\label{TG11}
For a paracompact Hausdorff space $X$ the games
$\Guf(\mathcal{O},\Omega)$ and $\SGuf(\mathcal{O},\Omega)$ are
equivalent.
\end{Th}

\begin{Th}[{\cite[Theorem 5.2]{dcna22-2}}]
\label{TG13}
For a metacompact space $X$ the games $\Guf(\mathcal{O},\Omega)$
and $\SSGf(\mathcal{O},\Omega)$ are equivalent.
\end{Th}

A similar observation for $\SSGC(\mathcal{O},\Omega)$ is the following.
\begin{Th}
\label{TG15}
For a paracompact space $X$ the games $\Guf(\mathcal{O},\Omega)$
and $\SSGC(\mathcal{O},\Omega)$ are equivalent.
\end{Th}

Combining Theorem~\ref{TG11}, Theorem~\ref{TG13} and Theorem~\ref{TG15}, we obtain the following.
\begin{Cor}
\label{CG1}
For a paracompact Hausdorff space $X$ the following games are
equivalent.
\begin{enumerate}[wide=0pt,label={\upshape(\arabic*)}]
  \item $\Guf(\mathcal{O},\Omega)$.
  \item $\SSGf(\mathcal{O},\Omega)$.
  \item $\SSGC(\mathcal{O},\Omega)$.
  \item $\SGuf(\mathcal{O},\Omega)$.
\end{enumerate}
\end{Cor}

In \cite[Example 5.4]{dcna22-2}, we proved that paracompactness
in Theorem~\ref{TG11} and metacompactness in Theorem~\ref{TG13}
are essential. In the following example we show that
paracompactness cannot be dropped also in Theorem~\ref{TG15}.
\begin{Ex}
\label{E4}
\emph{Paracompactness in Theorem~\ref{TG15} is essential}.\\
Let $X=X(\kappa)$ be the space as in Example~\ref{E6} with
$\kappa>\omega$. The space $X$ is Tychonoff $K$-starcompact but not paracompact. Since
$D\times\{\kappa^+\}$ is a closed discrete subset of $X$, $X$ is
not Scheepers and hence TWO has no winning strategy
in $\Guf(\mathcal{O},\Omega)$ on $X$.

We claim that TWO has a winning strategy in
$\SSGC(\mathcal{O},\Omega)$ on $X$. Let us define a strategy
$\sigma$ for TWO in $\SSGC(\mathcal{O},\Omega)$ on $X$ as
follows. In the $n$th inning, suppose that $\mathcal{U}_n$ is
the move of ONE in $\SSGC(\mathcal{O},\Omega)$. We get a compact
set $K_n\subseteq X$ such that $X=St(K_n,\mathcal{U}_n)$ since
$X$ is $K$-starcompact. Define
$\sigma(\mathcal{U}_1,\mathcal{U}_2,\dotsc,\mathcal{U}_n)=K_n$
as the response of TWO in $\SSGC(\mathcal{O},\Omega)$. Thus we
obtain a winning strategy $\sigma$ for TWO in
$\SSGC(\mathcal{O},\Omega)$ on $X$.
\end{Ex}

We now observe that in general the games
$\SGuf(\mathcal{O},\Omega)$ and $\SSGC(\mathcal{O},\Omega)$ are
not equivalent.
\begin{Ex}
\label{E1}
\emph{There exists a Hausdorff space in which TWO has a winning
strategy in $\SGuf(\mathcal{O},\Omega)$, but no winning strategy
in $\SSGC(\mathcal{O},\Omega)$.}\\
Let $X$ be the space as in Example~\ref{E3}. Observe that $X$ is
a Hausdorff space which is not star-$K$-Scheepers. It follows that TWO has no winning strategy in $\SSGC(\mathcal{O},\Omega)$ on $X$.

We now show that TWO has a winning strategy in
$\SGuf(\mathcal{O},\Omega)$ on $X$.
Define a strategy for TWO in $\SGuf(\mathcal{O},\Omega)$ as
follows. Choose $\mathcal{U}_1$ as the first move of ONE in
$\SGuf(\mathcal{O},\Omega)$. Then we can find a countable subset
$A$ of $P$ and a member $U_1$ of $\mathcal{U}_1$ such that
$U_p(A)\subseteq U_1$. By the construction of the topology on
$X$, we have $$(P\setminus A)\cup U_p(A)\subseteq
St(U_1,\mathcal{U}_1).$$ TWO responds in
$\SGuf(\mathcal{O},\Omega)$ by selecting
$\sigma(\mathcal{U}_1)=\{U_1\}$. For each $x_\alpha\in A$,
$$C_{x_\alpha}=\{x_\alpha\}\cup\{(x_\alpha,y_n) : n\in\mathbb{N}\}$$
is countable. Since $C=\cup_{x_\alpha\in A}C_{x_\alpha}$ is
countable, we enumerate $C$ as $\{a_{n+1} : n\in\mathbb{N}\}$.
Observe that $X=C\cup(P\setminus A)\cup U_p(A)$. Let the $n$th
move of ONE in $\SGuf(\mathcal{O},\Omega)$ be $\mathcal{U}_n$,
$n\geq 2$. Consider
$\sigma(\mathcal{U}_1,\mathcal{U}_2,\dotsc,\mathcal{U}_n)
=\{U_n\}$ as the response of TWO in $\SGuf(\mathcal{O},\Omega)$,
where $U_n\in\mathcal{U}_n$ with $a_n\in U_n$. Thus we define
the legitimate strategy $\sigma$ for TWO in
$\SGuf(\mathcal{O},\Omega)$ on $X$. It can be easily shown that
$\sigma$ is a winning strategy for TWO in
$\SGuf(\mathcal{O},\Omega)$ on $X$.
\end{Ex}

From the above example and Figure~\ref{dig2} it is also clear
that there exists a Hausdorff space $X$ such that TWO has a
winning strategy in $\SGuf(\mathcal{O},\Omega)$, but no winning
strategy in $\SSGf(\mathcal{O},\Omega)$ on $X$. Thus in general
the games $\SGuf(\mathcal{O},\Omega)$ and
$\SSGf(\mathcal{O},\Omega)$ are not equivalent.

Next example shows that in general the games
$\SSGC(\mathcal{O},\Omega)$ and $\SSGf(\mathcal{O},\Omega)$ are
not equivalent.
\begin{Ex}
\label{E2}
\emph{For any cardinal $\kappa>\omega$ there exists a Tychonoff space in which TWO has a winning strategy in $\SSGC(\mathcal{O},\Omega)$, but no winning strategy in $\SSGf(\mathcal{O},\Omega)$.}\\
Consider $X=Y(\kappa)$ as in Example~\ref{E5}. By \cite[Lemma 2.3]{sKH}, $X$ is a Tychonoff $K$-starcompact space. Also $X$ is not strongly star-Scheepers since it is not strongly star-Lindel\"{o}f. It now follows that TWO has no winning strategy in $\SSGf(\mathcal{O},\Omega)$ on $X$. If we proceed similarly as in Example~\ref{E4} by using $K$-starcompactness of $X$, it can be easily shown that TWO has a winning strategy in $\SSGC(\mathcal{O},\Omega)$ on $X$.
\end{Ex}

The relation between the winning strategies of the players ONE
and TWO in the games (for any space $X$) considered here can be
outlined into the following diagram (Figure~\ref{dig2}), where
the implication $G\rightarrow H$ holds if winning strategies for
TWO in $G$ produce winning strategies for TWO in $H$ as well as
winning strategies for ONE in $H$ produce winning strategies for
ONE in $G$ and the selection principle for $G$ implies the
selection principle for $H$, and also we use $G\dashrightarrow
H$ when winning strategies for TWO in $G$ need not produce
winning strategies for TWO in $H$ (for more details, see
\cite{dcna22, dcna22-2}).

\begin{figure}[h]
\begin{adjustbox}{keepaspectratio,center,
 max width=.8\textwidth,max height=.8\textheight}
\begin{tikzpicture}[scale=1,node distance=2.5cm and 2.5cm]
\node (k1) {$\SGuf(\mathcal{O},\Gamma)$};
\node (k2) [right=of k1] {$\SGuf(\mathcal{O},\Omega)$};
\node (k3) [right=of k2] {$\SGf(\mathcal{O},\mathcal{O})$};

\node (a1) [below=of k1] {$\SSGC(\mathcal{O},\Gamma)$};
\node (a2) [below=of k2] {$\SSGC(\mathcal{O},\Omega)$};
\node (a3) [below=of k3] {$\SSGC(\mathcal{O},\mathcal{O})$};

\node (b1) [below=of a1] {$\SSGf(\mathcal{O},\Gamma)$};
\node (b2) [below=of a2] {$\SSGf(\mathcal{O},\Omega)$};
\node (b3) [below=of a3] {$\SSGf(\mathcal{O},\mathcal{O})$};

\node (c1) [below=of b1] {$\Guf(\mathcal{O},\Gamma)$};
\node (c2) [below=of b2] {$\Guf(\mathcal{O},\Omega)$};
\node (c3) [below=of b3] {$\Gf(\mathcal{O},\mathcal{O})$};

\draw[->](k1) edge (k2);
\draw[->](k2) edge (k3);
\draw[->](a1) edge (k1);
\draw[->](a1) edge (a2);
\draw[->](a2) edge (a3);
\draw[->](a2) edge (k2);
\draw[->](a3) edge (k3);
\draw[->](b1) edge (a1);
\draw[->](b1) edge (b2);
\draw[->](b2) edge (b3);
\draw[->](b2) edge (a2);
\draw[->](b3) edge (a3);
\draw[->](c1) edge (c2);
\draw[->](c1) edge (b1);
\draw[->](c2) edge (c3);
\draw[->](c2) edge (b2);
\draw[->](c3) edge (b3);
\draw[->,dashed,thick,postaction={decorate,decoration={raise=2ex,text
along path,text align=center,text={|\tiny|{\cite[Example
3.3]{dcna22}}}}}](k1)[bend left=60] to (a1);
\draw[->,dashed,thick,postaction={decorate,decoration={raise=2ex,text
along path,text align=center,text={|\tiny|{\cite[Example
3.2]{dcna22}}}}}](a1)[bend left=60] to (b1);
\draw[->,dashed,thick,postaction={decorate,decoration={raise=2ex,text
along path,text align=center,text={|\tiny|{\cite[Example
3.1]{dcna22}}}}}](b1)[bend left=60] to (c1);
\draw[->,dashed,thick,postaction={decorate,decoration={raise=-2ex,text
along path,text align=center,text={|\tiny|{\cite[Example
3.3]{dcna22}}}}}](k3)[bend right=60] to (a3);
\draw[->,dashed,thick,postaction={decorate,decoration={raise=-2ex,text
along path,text align=center,text={|\tiny|{\cite[Example
3.2]{dcna22}}}}}](a3)[bend right=60] to (b3);
\draw[->,dashed,thick,postaction={decorate,decoration={raise=-2ex,text
along path,text align=center,text={|\tiny|{\cite[Example
3.1]{dcna22}}}}}](b3)[bend right=60] to (c3);
\draw[->,dashed,thick,postaction={decorate,decoration={raise=-2ex,text
along path,text align=center,text={|\tiny|Example
\getrefnumber{E1}}}}](k2)  [bend right=60] to (a2);
\draw[->,dashed,thick,postaction={decorate,decoration={raise=2ex,text
along path,text align=center,text={|\tiny|Example
\getrefnumber{E2}}}}](a2) to [bend left=60](b2);
\draw[->,dashed,thick,postaction={decorate,decoration={raise=-2ex,text
along path,text align=center,text={|\tiny|Example
\getrefnumber{E4}}}}](b2) to [bend right=60](c2);
\end{tikzpicture}
\end{adjustbox}
\caption{Diagram for winning strategies of player TWO}
\label{dig2}
\end{figure}
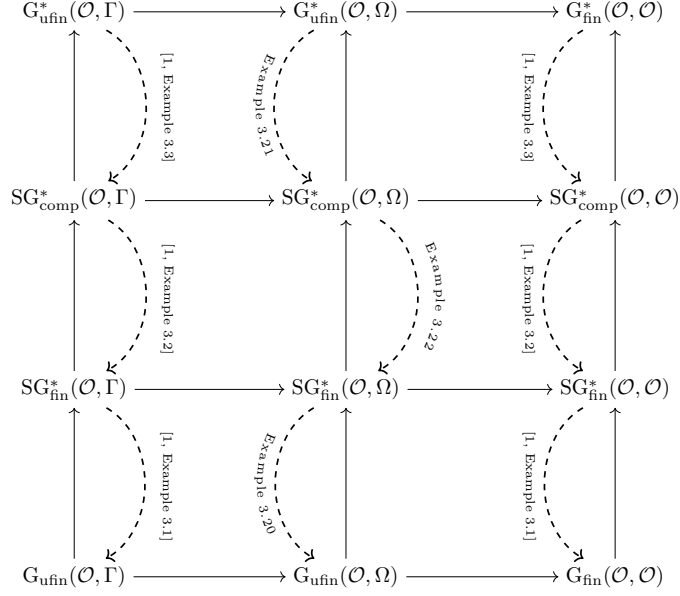

\section{Further observations on star-$K$-Scheepers spaces}
In this section we investigate preservation of the
star-$K$-Scheepers property under certain topological
operations. We also consider certain instances where such
preservation type of properties fail to hold. Now observe that
the star-$K$-Scheepers property is an invariant of continuous
mappings and is hereditary for clopen subsets.
\begin{Ex}
\label{E8}
\emph{There exists a Tychonoff star-$K$-Scheepers space having a regular-closed subset which is not star-$K$-Scheepers.}\\
Let $X=Y(\kappa)$ be the space as in Example~\ref{E5} with
$\kappa=\mathfrak{c}$ i.e. $$X=(\beta
D\times[0,\mathfrak{c}^+))\cup(D\times\{\mathfrak{c}^+\}).$$ The
space $X$ is Tychonoff star-$K$-Scheepers. Also consider $$Y=(\beta
D\times[0,\mathfrak{c}))\cup(D\times\{\mathfrak{c}\})$$ as a
subspace of $\beta D\times[0,\mathfrak{c}]$. We now show that
$Y$ is not star-$K$-Scheepers. For each $\alpha<\mathfrak{c}$
let $U_\alpha=\{d_\alpha\}\times(\alpha,\mathfrak{c}]$. Clearly
$U_\alpha$ is open in $Y$ and $U_\alpha\cap U_{\beta}=\emptyset$
if $\alpha\neq\beta$. Choose the sequence $(\mathcal{U}_n)_{n\in\mathbb{N}}$
of open covers of $Y$, where $$\mathcal{U}_n=\{U_\alpha :
\alpha<\mathfrak{c}\}\cup\{\beta D\times[0,\mathfrak{c})\}$$ for
each $n\in\mathbb{N}$. Let $(K_n)_{n\in\mathbb{N}}$ be any sequence of compact subspaces of $Y$. Since $\{(d_\alpha,\mathfrak{c}) : \alpha<\mathfrak{c}\}$ is a closed discrete subset of $Y$,
$K_n\cap\{(d_\alpha,\mathfrak{c}) : \alpha<\mathfrak{c}\}$ is
finite for each $n\in\mathbb{N}$. Thus for each $n\in\mathbb{N}$ there exists a $\alpha_n<\mathfrak{c}$ such that $$K_n\cap\{(d_\alpha,\mathfrak{c}) :
\alpha>\alpha_n\}=\emptyset.$$ Choose $\alpha_0=\sup\{\alpha_n :
n\in\mathbb{N}\}$ and so $\alpha_0<\mathfrak{c}$. It follows
that $$(\cup_{n\in\mathbb{N}}K_n)\cap\{(d_\alpha,\mathfrak{c}) :
\alpha>\alpha_0\}=\emptyset.$$ For each $n\in\mathbb{N}$ let $F_n=\{\alpha :
(d_\alpha,\mathfrak{c})\in K_n\}$. Clearly each $F_n$ is finite
and $K_n\setminus\cup_{\alpha\in F_n}U_\alpha$ is closed in
$K_n$ with $$K_n\setminus\cup_{\alpha\in
F_n}U_\alpha\subseteq\beta D\times[0,\mathfrak{c}).$$
Consequently $\pi(K_n\setminus\cup_{\alpha\in F_n}U_\alpha)$ is
a compact subspace of a countably compact space
$[0,\mathfrak{c})$, where $\pi:\beta
D\times[0,\mathfrak{c})\to[0,\mathfrak{c})$ is the projection.
Subsequently for each $n\in\mathbb{N}$ we get a $\beta_n<\mathfrak{c}$ such
that $$\pi(K_n\setminus\cup_{\alpha\in
F_n}U_\alpha)\cap(\beta_n,\mathfrak{c})=\emptyset.$$ Choose
$\beta_0=\sup\{\beta_n : n\in\mathbb{N}\}$ and observe that
$\beta_0<\mathfrak{c}$. If we pick a
$\gamma>\max\{\alpha_0,\beta_0\}$, then $U_\gamma\cap
K_n=\emptyset$ for all $n\in\mathbb{N}$. Notice that $U_\gamma$ is the only
member of $\mathcal{U}_n$ containing $(d_\gamma,\mathfrak{c})$
for all $n\in\mathbb{N}$. Thus
$$St((d_\gamma,\mathfrak{c}),\mathcal{U}_n)\cap K_n=\emptyset$$
for all $n\in\mathbb{N}$. It now follows that $Y$ is not star-$K$-Scheepers
as $F=\{(d_\gamma,\mathfrak{c})\}$ is a finite subset of $Y$
with $$F\nsubseteq St(K_n,\mathcal{U}_n)$$ for all $n\in\mathbb{N}$.

Let $f:D\times\{\mathfrak{c}^+\}\to D\times\{\mathfrak{c}\}$ be
a bijection and $Z$ be the quotient image of the topological sum
$X\oplus Y$ obtained by identifying $(d_\alpha,\mathfrak{c}^+)$
of $X$ with $f(d_\alpha,\mathfrak{c}^+)$ of $Y$ for every
$\alpha<\mathfrak{c}$. Let $q:X\oplus Y\to Z$ be the quotient
map. Thus $q(Y)$ is a regular-closed subset of $Z$, but $q(Y)$
is not star-$K$-Scheepers as it is homeomorphic to $Y$. Finally
from \cite[Example 2.11]{sKH}, $Z$ is star-$K$-Scheepers.
\end{Ex}

Assuming $\mathfrak{d}\leq\mathfrak{r}$, one can use \cite[Theorem 2.5]{FPM} and \cite[Theorem 3.9]{coc2} to find two Scheepers sets of reals whose union is not Scheepers. Thus in view of Proposition~\ref{T1}, there exist two star-$K$-Scheepers spaces whose union is not star-$K$-Scheepers. Now take $X=\cup_{k\in\mathbb{N}}X_k$, where $X_k\subseteq X_{k+1}$ for all $k\in\mathbb{N}$. If each $X_k$ is star-$K$-Scheepers, then $X$ is also star-$K$-Scheepers. The converse of this assertion is not true.  For example, let $X$ be the Isbell-Mr\'{o}wka space $\Psi(\mathcal{A})$ \cite{Mrowka} with $|\mathcal{A}|=\omega_1$, under the assumption that $\omega_1<\mathfrak{d}$. By \cite[Corollary 3.19(1)]{dcna22-2}, $X$ is strongly star-Scheepers hence $X$ is star-$K$-Scheepers. We can write $X=\cup_{n\in\mathbb{N}}X_n$, where $$X_n=\mathcal{A}\cup\{1,2,\dotsc,n\}$$ for each $n\in\mathbb{N}$. Since each $X_n$ is discrete and $|X_n|=\omega_1$, it follows that $X_n$ is not star-$K$-Scheepers for any $n\in\mathbb{N}$.

Next we turn to consider Alexandroff duplicate $AD(X)$ of a space $X$ \cite{AD,Engelking}, which is defined as follows. $AD(X)=X\times\{0,1\}$; each point of $X\times\{1\}$ is isolated and a basic neighbourhood of $(x,0)\in X\times\{0\}$ is a set of the form $(U\times\{0\})\cup((U\times\{1\})\setminus\{(x,1)\})$, where $U$ is a neighbourhood of $x$ in $X$.

\begin{Th}[{\cite[Theorem 3.32]{dcna22-2}}]
\label{T4}
The following assertions are equivalent.
\begin{enumerate}[wide=0pt,label={\upshape(\arabic*)},leftmargin=*]
  \item $X$ is Scheepers.
  \item $AD(X)$ is Scheepers.
\end{enumerate}
\end{Th}

\begin{Cor}
\label{C5}
If $X$ is Scheepers, then $AD(X)$ is star-$K$-Scheepers.
\end{Cor}

\begin{Ex}
\label{E9}
\emph{There exists a Tychonoff star-$K$-Scheepers space $X$ such that $AD(X)$ is not star-$K$-Scheepers.}\\
Let $X=Y(\kappa)$ be the space as in Example~\ref{E5} with
$\kappa>\omega$. The space $X$ is Tychonoff star-$K$-Scheepers. But $AD(X)$ is not star-$K$-Scheepers. Indeed, since $D\times\{\kappa^+\}$ is a closed discrete subset of $X$, $$(D\times\{\kappa^+\})\times\{1\}$$ is a clopen discrete subset of $AD(X)$ which is not star-$K$-Scheepers, hence $AD(X)$ has not star-$K$-Scheepers property since it should be preserved under clopen subsets.
\end{Ex}

It is natural to ask the following.
\begin{Prob}
\label{Q2}
Does there exist a space $X$ such that $AD(X)$ is
star-$K$-Scheepers, but $X$ is not star-$K$-Scheepers?
\end{Prob}

The following result indicates that the answer to the above
problem is not affirmative.
\begin{Th}
\label{T3}
If $AD(X)$ is star-$K$-Scheepers, then $X$ is also
star-$K$-Scheepers.
\end{Th}
\begin{proof}
Consider a sequence $(\mathcal{U}_n)_{n\in\mathbb{N}}$ of open covers of $X$. We
define a sequence $(\mathcal{V}_n)_{n\in\mathbb{N}}$ of open covers of $AD(X)$, where for each $n\in\mathbb{N}$, $$\mathcal{V}_n=\{U\times\{0,1\} :
U\in\mathcal{U}_n\}.$$ Since $AD(X)$ is star-$K$-Scheepers,
there exists a sequence  $(K_n)_{n\in\mathbb{N}}$ of compact subspaces of $AD(X)$ such that $\{St(K_n,\mathcal{V}_n) : n\in\mathbb{N}\}$ is an
$\omega$-cover of $AD(X)$. For each $n\in\mathbb{N}$ let $$C_n=\{x\in X :
\text{either $(x,0)\in K_n$ or $(x,1)\in K_n$}\}.$$ Observe that
for each $n\in\mathbb{N}$, $C_n$ is a compact subspace of $X$. The sequence
$(C_n)_{n\in\mathbb{N}}$ witnesses that the space is star-$K$-Scheepers.
\end{proof}

The extent of a Lindel\"{o}f (and hence Scheepers) space must be
countable. But in case of star-$K$-Scheepers spaces the extent
can be arbitrarily large (see Examples~\ref{E6} and \ref{E5}). However we obtain the following.
\begin{Th}
\label{T2}
If $X$ is a $T_1$ space such that $AD(X)$ is star-Lindel\"{o}f,
then $e(X)\leq\omega$.
\end{Th}
\begin{proof}
Suppose by contradiction that $e(X)\geq\omega_1$. It follows that
there exists a discrete closed subset $D$ of $X$ such that
$|D|\geq\omega_1$. It is easy to see that $D\times\{1\}$ is a
clopen subset of $AD(X)$ and since each point of $D\times\{1\}$
is isolated, it is also discrete. Thus $AD(X)$ is not
star-Lindel\"{o}f as the star-Lindel\"{o}f property is preserved
under clopen subsets, which is a contradiction. Consequently
$e(X)\leq\omega$.
\end{proof}

We obtain the following result of Song as a consequence of the
preceding theorem.
\begin{Cor}[{\!\cite[Theorem 2.6]{rsM}}]
\label{C1}
If $X$ is a $T_1$ space such that $AD(X)$ is star-Menger, then
$e(X)\leq\omega$.
\end{Cor}

\begin{Cor}[{\!\cite[Theorem 2.13]{RSKM}}]
\label{C3}
If $X$ is a $T_1$ space such that $AD(X)$ is star-$K$-Menger,
then $e(X)\leq\omega$.
\end{Cor}

\begin{Cor}
\label{C2}
If $X$ is a $T_1$ space such that $AD(X)$ is star-$K$-Scheepers,
then $e(X)\leq\omega$.
\end{Cor}

In view of the above corollary, we ask the following question.
\begin{Prob}
\label{Pb1}
Is the space $AD(X)$ of a star-$K$-Scheepers space $X$ with
$e(X)\leq\omega$ also star-$K$-Scheepers?
\end{Prob}

Now observe that preimage of a Tychonoff star-$K$-Scheepers
space under a closed 2-to-1 continuous mapping may not be
star-$K$-Scheepers. If $Y=Y(\kappa)$ is the space as in
Example~\ref{E5} with $\kappa>\omega$, then $Y$ is Tychonoff
star-$K$-Scheepers. By Example~\ref{E9}, $X=AD(Y)$ is not
star-$K$-Scheepers. Besides, the projection $p:X\to Y$ is a
closed 2-to-1 continuous mapping. However we obtain the
following.
\begin{Th}
\label{T11}
If $f:X\to Y$ is an open perfect mapping from a space $X$ onto a
star-$K$-Scheepers space $Y$, then $X$ is star-$K$-Scheepers.
\end{Th}
\begin{proof}
Let $(\mathcal{U}_n)_{n\in\mathbb{N}}$ be a sequence of open covers of $X$ and
$y\in Y$. Since $f^{-1}(y)$ is compact, for each $n\in\mathbb{N}$ there
exists a finite subset $\mathcal{V}_n^y$ of $\mathcal{U}_n$ such
that $f^{-1}(y)\subseteq\cup\mathcal{V}_n^y$ and $f^{-1}(y)\cap
U\neq\emptyset$ for each $U\in\mathcal{V}_n^y$. Since $f$ is
closed, there exists an open set $U_y^{(n)}$ in $Y$ containing
$y$ such that $f^{-1}(U_y^{(n)})\subseteq\cup\mathcal{V}_n^y$.
Also by the openness of $f$, we can find an open set $V_y^{(n)}$
in $Y$ containing $y$ such that $$V_y^{(n)}\subseteq\cap\{f(U) :
U\in\mathcal{V}_n^y\}$$ and $f^{-1}(V_y^{(n)})\subseteq
f^{-1}(U_y^{(n)})$. Thus we obtain a sequence $(\mathcal{V}_n)_{n\in\mathbb{N}}$
of open covers of $Y$, where for each $n\in\mathbb{N}$, $\mathcal{V}_n=\{V_y^{(n)} : y\in Y\}$. Apply the
star-$K$-Scheepers property of $Y$ to $(\mathcal{V}_n)_{n\in\mathbb{N}}$ to
obtain a sequence $(K_n^\prime)_{n\in\mathbb{N}}$ of compact subspaces of $Y$ such that $\{St(K_n^\prime,\mathcal{V}_n) : n\in\mathbb{N}\}$ is an
$\omega$-cover of $Y$. For each $n\in\mathbb{N}$ choose $K_n=f^{-1}(K_n^\prime)$. Clearly $(K_n)_{n\in\mathbb{N}}$ is a sequence of compact subspaces of $X$ since $f$ is a perfect mapping. We now show that $\{St(K_n,\mathcal{U}_n) : n\in\mathbb{N}\}$ is an $\omega$-cover of $X$. Let $F$ be a finite subset of $X$. There exists a $n_0\in\mathbb{N}$ such that $f(F)\subseteq St(K_{n_0}^\prime,\mathcal{V}_{n_0})$. An easy verification gives $$f^{-1}(St(K_{n_0}^\prime,\mathcal{V}_{n_0}))\subseteq
St(K_{n_0},\mathcal{U}_{n_0}).$$ It follows that $F\subseteq
St(K_{n_0},\mathcal{U}_{n_0})$ and consequently $\{St(K_n,\mathcal{U}_n) : n\in\mathbb{N}\}$ is an $\omega$-cover of $X$. This completes the proof.
\end{proof}

Since the star-$K$-Scheepers property is preserved under
countable increasing unions, we obtain the following.
\begin{Cor}
\label{C4}
If $X$ is a star-$K$-Scheepers space and $Y$ is a
$\sigma$-compact space, then $X\times Y$ is star-$K$-Scheepers.
\end{Cor}

In the next example we show that product of two
star-$K$-Scheepers spaces may not be star-$K$-Scheepers.
\begin{Ex}
\label{E15}
\emph{There exist two countably compact spaces such that their
product is not star-Lindel\"{o}f (hence not
star-$K$-Scheepers).}\\
In \cite{Song}, the following construction is considered and in \cite{SX}, it is proved that the space $X\times Y$ is not star-Lindel\"{o}f.
Let $D$ be the discrete space with cardinality $\mathfrak{c}$.
Let $X=\cup_{\alpha<\omega_1}E_\alpha$ and
$Y=\cup_{\alpha<\omega_1}F_\alpha$, where $E_\alpha$ and
$F_\alpha$ are subsets of $\beta D$ such that
\begin{enumerate}[wide=0pt,label={\upshape(\arabic*)},leftmargin=*]
  \item $E_\alpha\cap F_\beta=D$ if $\alpha\neq\beta$;
  \item $|E_\alpha|\leq\mathfrak{c}$ and
      $|F_\beta|\leq\mathfrak{c}$;
  \item every infinite subset of $E_\alpha$ (resp. $F_\alpha$)
      has an accumulation point in $E_{\alpha+1}$ (resp.
      $F_{\alpha+1}$).
\end{enumerate}

Thus both $X$ and $Y$ are countably compact but $X\times Y$ is not star-$K$-Scheepers.
\end{Ex}

Next we observe that Corollary~\ref{C4} does not hold if $X$ is
star-$K$-Scheepers and $Y$ is Lindel\"{o}f.
\begin{Ex}
\label{E17}
\emph{There exist a countably compact space $X$ and a
Lindel\"{o}f space $Y$ such that $X\times Y$ is not
star-Lindel\"{o}f (hence not star-$K$-Scheepers).}\\
Let $X=[0,\omega_1)$ be the space with the usual order topology.
Also define a topology on $Y=[0,\omega_1]$ as follows. Each
point $\alpha<\omega_1$ is isolated and a set $U$ containing
$\omega_1$ is open if and only if $Y\setminus U$ is countable.
Clearly $X$ is countably compact and $Y$ is Lindel\"{o}f. It is proved in \cite[Example 3.8]{Singh} that $X\times Y$ is not star-Lindel\"{o}f.
\end{Ex}

\section{Open Problems}
Consider an almost disjoint family $\mathcal{A}\subseteq P(\mathbb{N})$ with $|\mathcal{A}|=\mathfrak{b}$. Under the hypothesis $\mathfrak{b}=\aleph_1<\mathfrak{d}$, $\Psi(\mathcal{A})$ is not star-Hurewicz by \cite[Theorem 2.4]{CASC} (hence not star-$K$-Hurewicz) but strongly star-Scheepers by \cite[Corollary 3.19]{dcna22-2} (hence star-$K$-Scheepers). Thus a star-$K$-Scheepers space is not necessarily either Scheepers or star-$K$-Hurewicz.

In addition to Problem~\ref{P6}, the following questions seem to be interesting.
\begin{Prob}
\label{Pb3}
Is there a star-$K$-Scheepers space which is neither strongly star-Scheepers nor star-$K$-Hurewicz?
\end{Prob}

\begin{Prob}
\label{Pb4}
Is there a star-$K$-Menger space which is neither Menger nor star-$K$-Scheepers?
\end{Prob}

\begin{Prob}
\label{Pb5}
Is there a star-$K$-Menger space which is neither strongly star-Menger nor star-$K$-Scheepers?
\end{Prob}

\noindent{\bf Acknowledgement:} The authors are thankful to the Referee(s) for several valuable comments and suggestions that considerably improved the presentation of the paper.

\end{document}